\documentclass[letterpaper, reqno, 11pt]{amsart}
\usepackage[english]{babel}
\usepackage{graphicx,color}
\usepackage[all]{xy}
\usepackage{amsfonts,dsfont}
\usepackage{amsmath, amsthm, amssymb}
\usepackage{mathrsfs}
\usepackage{enumerate, url}
\usepackage{fancyhdr}
\usepackage{ifthen,srcltx}
\usepackage{comment}

\newcommand{\bbR}{\ensuremath{\mathbb{R}}}

\newcommand{\bbZ}{\ensuremath{\mathbb{Z}}}

\newcommand{\bbN}{\ensuremath{\mathbb{N}}}

\newcommand{\bbone}{\ensuremath{\mathds{1}}}

\newcommand{\calR}{\ensuremath{\mathcal{R}}}

\DeclareMathOperator{\Sch}{Sch}

\DeclareMathOperator{\supp}{supp}

\DeclareMathOperator{\Sub}{Sub}
\def\bs{\backslash}
\numberwithin{equation}{section}

\newtheorem{thm}{Theorem}[section]
\newtheorem{prop}[thm]{Proposition}
\newtheorem{lem}[thm]{Lemma}
\newtheorem{cor}[thm]{Corollary}
\theoremstyle{definition}
\newtheorem{dfn}[thm]{Definition}

\newtheorem{rmk}[thm]{Remark}

\newtheorem{ex}[thm]{Example}

\newcommand{\showcomments}{yes}

\newsavebox{\commentbox}
\newenvironment{mycomment}%
{\ifthenelse{\equal{\showcomments}{yes}}%
{\footnotemark
        \begin{lrbox}{\commentbox}
        \begin{minipage}[t]{1.25in}\raggedright\sffamily\tiny
        \footnotemark[\arabic{footnote}]}
{\begin{lrbox}{\commentbox}}}
{\ifthenelse{\equal{\showcomments}{yes}}
{\end{minipage}\end{lrbox}\marginpar{\usebox{\commentbox}}}
{\end{lrbox}}}

\begin{document}

\title{Co-spectral radius of intersections} 

\author[M. Fraczyk]{Mikolaj Fraczyk}
\address{Department of Mathematics \\ University of Chicago \\ Chicago, IL 60637 \\ USA}

\author[W. van Limbeek]{Wouter van Limbeek}
\address{Department of Mathematics, Statistics, and Computer Science \\ 
                 University of Illinois at Chicago \\
                 Chicago, IL 60647 \\
                 USA}
                 
\date{\today}

\begin{abstract}
We study the behavior of the co-spectral radius of a subgroup $H$ of a discrete group $\Gamma$ under taking intersections. Our main result is that the co-spectral radius of an invariant random subgroup does not drop upon intersecting with a deterministic co-amenable subgroup. As an application, we find that the intersection of independent co-amenable invariant random subgroups is co-amenable.
\end{abstract}

\maketitle


\section{Introduction}
\label{sec:intro}
Let $\Gamma$ be a countable group and let $S$ be a finite symmetric subset of $\Gamma$. The \textbf{co-spectral radius} of a subgroup $H\subset \Gamma$ (with respect to $S$) is defined as the norm of the operator $M\colon L^2(\Gamma/H)\to L^2(\Gamma/H)$:
\[ M\phi(\gamma H)=\frac{1}{|S|}\sum_{s\in S}\phi(s\gamma H),\quad \rho(\Gamma/H):=\|M\|.\]
The groups with co-spectral radius $1$ for every choice of $S$ are called {\bf co-amenable}. If the group $\Gamma$ is finitely generated, one needs only to verify that the co-spectral radius is $1$ for some generating set $S$. 

In this paper we investigate the behavior of co-spectral radius under intersections. For general subgroups $H_1,H_2\subset \Gamma$ there is not much that can be said about $\rho(\Gamma/ (H_1\cap H_2))$ other than the trivial inequality \[\rho(\Gamma/ (H_1\cap H_2)) \leq \min\{\rho(\Gamma/H_1),\rho(\Gamma/H_2)\}.\] The problem of finding  lower bounds on the co-spectral radius of an intersection is even more dire, as there are examples of non-amenable $\Gamma$ with two co-amenable subgroups $H_1,H_2$ with trivial intersection (see Example \ref{ex:Wreath}). 
However, when considering all conjugates simultaneously, we have the following elementary lower bound on the co-spectral radius of an intersection. Here and for the remainder of the paper, we write $H^g:=g^{-1}Hg$. 
\begin{thm}\label{thm:DetInt}
Let $\Gamma$ be a finitely generated group and let $S$ be a  finite symmetric generating set. Let $H_1,H_2$ be subgroups of $\Gamma$ and assume that $H_1$ is co-amenable. Then 
\[\sup_{g\in \Gamma} \rho(\Gamma/ (H_1\cap H_2^g))= \rho(\Gamma/H_2).\]
\end{thm}
The supremum over all conjugates in the statement of the theorem is in fact necessary, as shown by Example \ref{ex:Wreath}. However, in the presence of invariance, this can be improved upon: E.g. if $H_2=N$ is normal, Theorem \ref{thm:DetInt} immediately implies that $\rho(\Gamma/ (H_1\cap N))= \rho(\Gamma/N)$. Our aim is to generalize this to invariant random subgroups.  

An \textbf{invariant random subgroup} of $\Gamma$ (see \cite{AGV14}) is a random subgroup of $\Gamma$ whose distribution is invariant under conjugation. Invariant random subgroups simultaneously generalize the notion of finite index subgroups and normal subgroups. They have proven to be very useful tools in measured group theory (see for example \cite{7SamN,Bowen2014,BLT}).
Many results on invariant random subgroups are obtained as generalizations of statements previously known for normal subgroups. We follow this tradition and show that one can remove the supremum in Theorem \ref{thm:DetInt} when $H_2$ is an IRS. The co-spectral radius of a subgroup $H$ is invariant under conjugation of $H$ by elements of $\Gamma$ and defines a measurable function\footnote{The measurability follows from the fact that $\rho(H\bs\Gamma)=\lim_{R\to\infty}\sup_{\supp f\subset B(R)}\frac{\langle Mf,f\rangle}{\|f\|^2}$ where $B(R)$ is the $R$-ball around $H$ in $H\bs\Gamma$ and $f$ runs over non-zero functions supported in $B(R)$. The ball $B(R)$ depends measurably on $H$, so the co-spectral radius must be measurable as well.} on the space of subgroups of $\Gamma$. It follows that the co-spectral radius of an ergodic IRS $H$ is constant a.s. and therefore $\rho(\Gamma/H)$ is well-defined. We say an IRS is co-amenable if it is co-amenable a.s. Our main result is:

\begin{thm}
Let $\Gamma$ be a countable with a finite symmetric subset $S$. Let $H_1\subset \Gamma$ be a deterministic co-amenable subgroup and let $H_2$ be an ergodic invariant random subgroup of $\Gamma$. Then 
\[ \rho(\Gamma/ (H_1\cap H_2))=\rho(\Gamma/ H_2)\] almost surely.
\label{thm:main-spectral}
\end{thm}

This result was inspired by a question of Alex Furman, asking whether the intersection of a co-amenable IRS'ses remains co-amenable. A positive answer follows from Theorem \ref{thm:main-spectral} applied in the case when both $H_1,H_2$ are co-amenable IRS'es.
\begin{cor} Let $\Gamma$ be a countable group and let $H_1,H_2$ be independent co-amenable invariant random subgroups. Then the intersection $H_1\cap H_2$ is co-amenable.
\label{cor:intersectIRS}
\end{cor}
\begin{rmk} The independence assumption in the above corollary is necessary (see Example \ref{ex:independence}).\end{rmk}
Combined with Cohen-Grigorchuk's co-growth formula \cite{Cohen,Grig} and Gekhtman-Levit's lower bound on the critical exponent of an IRS of a free group \cite[Thm 1.1]{GL19}, Theorem \ref{thm:main-spectral} yields the following corollary on the critical exponents of subgroups of the free group.
\begin{cor}
Let $F_d$ be a free group on $d\geq 2$ generators and let $S$ be the standard symmetric generating set. Write $\delta(H)$ for the critical exponent of a subgroup $H$. Suppose $H_1$ is co-amenable and $H_2$ is an ergodic IRS. Then 
\[ \delta(H_1\cap H_2)=\delta(H_2),\] almost surely. 
\end{cor}




\subsection{Outline of the proof} 
\label{sec:outline}
We outline the proof of Theorem \ref{thm:main-spectral}. For the sake of simplicity we restrict to the case when both $H_1,H_2$ are co-amenable. We realize $H_1$ and $H_2$ as stabilizers of $\Gamma$-actions on suitable spaces $X_1$ and $X_2$. Since $H_1$ is a (deterministic) subgroup, we can take $X_1=\Gamma / H_1$. On the other hand, $H_2$ is an ergodic IRS so $X_2$ is a probability space with an ergodic measure-preserving action of $\Gamma$. The intersection $H_1\cap H_2$ is then the stabilizer of a point in $X_1\times X_2$ that is deterministic in the first variable and random in the second. Using an analogue of the Rokhlin lemma, we can find a positive measure subset $E$ of $X_2$ that locally approximates the coset space $H_2\bs\Gamma$. The product of $E$ with $X_1$ will locally approximate the coset space of $(H_1\cap H_2)\bs \Gamma$. The co-amenability of $H_2$ means that the set $E$ contains a subset $P$ that is nearly $\Gamma$-invariant. The product of such a set with a F{\o}lner set $F$ in $X_1$ should be a nearly $\Gamma$-invariant set in $X_1\times E$, and hence witnesses the amenability of the coset graph $(H_1\cap H_2)\bs \Gamma$. The latter implies that $H_1\cap H_2$ is co-amenable. 

The actual proof is more complicated because if the product system is not ergodic, one has to show the product set is nearly invariant after restriction to each ergodic component and not just on average. Otherwise we can only deduce the bound from Theorem \ref{thm:DetInt} using a supremum over all conjugates. Obtaining control on each ergodic component is a key part of the proof where we actually use the invariance of $H_2$. 



To prove Theorem \ref{thm:main-spectral} for the co-spectral radius, one should replace the F{\o}lner set in $X_2$ with a function $f_2$ that (nearly) witnesses the fact that $\rho(\Gamma/H_2)=\lambda_2$, and adapt the remainder of the proof accordingly.


\subsection*{Outline of the paper} Section \ref{sec:recap} contains background material. In Section \ref{sec:deterministic} we prove the deterministic bound on spectral radius given by Theorem \ref{thm:DetInt}. Next, in Section \ref{sec:whf}, we rephrase co-spectral radius of the (discrete) orbits in terms of \emph{embedded spectral radius} on a (continuous) measure space. 

\subsection*{Acknowledgments} The authors thank Gabor Pete and Tianyi Zheng for showing us a related problem for intersections of percolations on $\bbZ$.  We thank Alex Furman for suggesting the problem  as well as for helpful discussions. MF thanks Miklos Abert for useful discussions. We thank the anonymous referee for valuable comments. We thank the University of Illinois at Chicago for providing support for a visit by MF. MF was partly supported by ERC Consolidator Grant 648017. WvL is supported by NSF DMS-1855371. 

\section{Background}
\label{sec:recap}

\subsection{Co-amenability}
\label{sec:coamenable}
Let $\Gamma$ be a finitely generated group and let $S$ be a finite symmetric set of generators. A subgroup $H$ of $\Gamma$ is called \textbf{ co-amenable} if the Schreier graph $\Sch(H\bs \Gamma,S)$ is amenable, i.e for any $\varepsilon>0$ and any $S$ there exists a set $F\subset H\bs \Gamma$ such that $|F\Delta FS|\leq \varepsilon |F|$. Such sets will be called $\varepsilon$-F{\o}lner sets. Alternatively, a subgroup $H$ is co-amenable if and only if the representation $\ell^2(\Gamma/H)$ has almost invariant vectors, or that  $\|M\|_{L^2(H\bs \Gamma)}=1.$ 

\subsection{Invariant random subgroups}
\label{sec:IRS} Let $\Sub_\Gamma$ be the space of subgroups of $\Gamma$, equipped with the topology induced from $\{0,1\}^\Gamma$. An \textbf{invariant random subgroup} is a probability measure $\mu\in \mathcal{P}(\Sub_\Gamma)$ which is invariant under conjugation of $\Gamma$. An IRS is called {\bf co-amenable} if 
    $$\mu(\{H\in\Sub_\Gamma| H \textrm{ is co-amenable}\})=1.$$
Similarly, we say that an IRS $H$ has {\bf co-spectral radius at least} $\lambda$ if $\rho(H\bs\Gamma)\geq \lambda$ almost surely. 

For any action $\Gamma\curvearrowright X$ and $x\in X$ write $\Gamma_x$ for the stabilizer of $x$. Every IRS can be realized as a stabilizer of a random point in a probability measure preserving system:
\begin{thm}[Abert-Glasner-Virag \cite{AGV14}] For every IRS $\mu$, there exists a standard Borel probability space $(X,\nu)$ and a Borel p.m.p. $\Gamma$-action on $(X,\nu)$ such that $\mu=\int_{X}\delta_{\Gamma_x} \, d\nu(x).$ 
\end{thm}

\subsection{Ergodic decomposition of infinite measures}\label{sec:decomp}
The material in this subsection is well-known to experts but difficult to locate in the literature. Our goal is to construct an ergodic decomposition for measure-preserving actions of countable groups on spaces with an infinite measure. We deduce this from the corresponding result for nonsingular actions on probability spaces:

\begin{thm}[{Greschonig-Schmidt \cite[Thm 1]{ergodic-decomp}}] Let $\Gamma$ be a countable group and let $\Gamma\curvearrowright (X,\Sigma_X, \nu)$ be a nonsingular Borel action on a standard Borel probability space. Then there exist a standard Borel probability space $(Z,\Sigma_Z, \tau)$ and a family of quasi-invariant, ergodic, pairwise mutually singular probability measures $\{\nu_z\}_{z\in Z}$ with the same Radon-Nikodym cocycle as $\nu$, and  such that for every $B\in\Sigma_X$, we have
            \begin{equation}
                \nu(B)=\int_Z \nu_z(B) d\tau(z).
                \label{eq:disintegrate}
            \end{equation}
\label{thm:greschonig-schmidt}            
\end{thm}
As an application we have:
\begin{cor}
Let $\Gamma$ be a countable group and let $\Gamma\curvearrowright (X_1,\Sigma_{X_1}, \nu_1)$ and $\Gamma\curvearrowright (X_2,\Sigma_{X_2}, \nu_2)$ be measure-preserving Borel actions on standard Borel spaces. Suppose that $(X_1,\nu_1)$ is ergodic and that $\nu_2(X_2)=1.$ Then there exists a standard Borel probability space $(Z,\Sigma_Z,\tau)$ and a family of $\Gamma$-invariant, ergodic, pairwise mutually singular measures $\{\nu_z\}_{z\in Z}$ on $X_1\times X_2$ such that for every $B\in\Sigma_{X_1\times X_2}$, we have
    \begin{equation}
        \nu(B)=\int_Z \nu_z(B)d\tau(z).
        \label{eq:disintegrate-us}
    \end{equation}
Moreover, for every measurable set $F\subset X_1$ and $z\in Z$ we have $$\nu_z(F\times X_2)=\nu_1(F).$$
\label{cor:erg-decomp}
\end{cor}
  \begin{proof} Fix a countably-valued Borel function $w\colon X_1\to\mathbb R_{>0}$, such that $\int_{X_1} w \, d\nu_1=1$. Write $w=c_i$ on the set $A_i$, where $\{A_i\}$ is a Borel partition of $X_1$. 
  
  Then $w(x_1)d\nu(x_1)d\nu_2(x_2)$ is a $\Gamma$-quasi-invariant probability measure on $X_1\times X_2$ with Radon-Nikodym cocycle $dw(x_1,x_2,\gamma)=\frac{w(\gamma x_1)}{w(x_1)}$. Let $(Z,\Sigma_Z,\tau)$, $z\mapsto  (w\nu)_z$, be its ergodic decomposition as provided by Theorem \ref{thm:greschonig-schmidt}. 
Now pass back from $w(x_1)d\nu_1(x_1)d\nu_2(x_2)$ to $\nu_1\times\nu_2$ by setting $$d\nu_z(x_1,x_2):=w(x_1)^{-1}d(w\nu)_z.$$ Since $\{(w\nu)_z\}_z$ are ergodic and pairwise mutually singular, the same is true for $\{\nu_z\}_z$. Since $(w\nu)_z$ have Radon-Nikodym cocycle $dw$, the measures $\nu_z$ are $\Gamma$-invariant. 


It is easy to verify that Equation \eqref{eq:disintegrate} implies the corresponding Equation \eqref{eq:disintegrate-us}.

Finally, to satisfy the last identity, choose a positive measure subset $F\subset X_1$ and renormalize $\nu_z$ and $\tau$ as follows: $$\nu_z\mapsto \frac{\nu_1(F)}{\nu_z(F\times X_2)}\nu_z,\quad d\tau(z)\to \frac{\nu_z(F\times X_2)}{\nu_1(F)}d\tau(z).$$ By ergodicity of $\nu_1$, this normalization does not depend on the choice of $F$.
\end{proof}

\subsection{Ergodic theory of equivalence relations} 
\label{sec:relations}
Let $(X,\nu)$ be a probability measure space and let $\varphi_i:U_i\to X$ be a finite family of non-singular measurable maps defined on subsets $U_i$ of $X$. The triple $(X,\nu,(\varphi_i)_{i\in I})$ is called a \textbf{graphing}. We assume that $(\varphi_i)_{i\in I}$ is \textbf{symmetric}, i.e. for each $i\in I$ the map $\varphi_i^{-1}\colon \varphi_i(U_i)\to U_i$ is also in the set $(\varphi_i)_{i\in I}$. A graphing is {\bf finite} if the index set $I$ is finite.

\begin{rmk} In our applications of this theory, $X$ will be a finite measure subset (not necessarily invariant) of a measure-preserving action of $\Gamma$, equipped with the graphing corresponding to a finite symmetric generating set $S$ of $\Gamma$, and $\nu$ will be the restricted measure or the restriction of an ergodic component. \end{rmk}

Let $\mathcal R$ be the orbit equivalence relation generated by the maps $(\varphi_i)_{i\in I}$. A measured graphing yields a random graph in the following way: For every $x\in X$, let $\mathcal G_x$ be the graph with vertex set given by the equivalence class $[x]_{\mathcal R}$ and place  an edge between $y,z\in [x]_{\mathcal R}$ whenever $z=\varphi_i(y)$ for some $i\in I$ (multiple edges are allowed). The graphs $\mathcal  G_x$ have degrees bounded by $|I|$ and are undirected since $(\varphi_i)_{i\in I}$ is symmetric. If we choose a $\nu$-random point $x$, the resulting graph $\mathcal G_x$ is a random rooted graph. The properties of $\mathcal G_x$ will depend on the graphing. For example, if the graphing consists of measure preserving maps then the resulting random graph is unimodular (see \cite{AL}). 

Suppose from now on the graphing is measure-preserving. Then the \textbf{mass transport principle} \cite{AL} asserts that for any measurable function $K:\mathcal R\to \bbR$, we have
\begin{equation}\label{eq:MTPsub}
\int_X\left( \sum_{x'\in [x]_{\mathcal R}}K(x,x')\right) \, d\nu(x)=\int_X\left( \sum_{x\in [x']_{\mathcal R}}K(x,x')\right) \, d\nu(x').
\end{equation}

\section{Co-spectral radius for deterministic intersections}\label{sec:deterministic}
In this section, we prove Theorem \ref{thm:DetInt} that gives the elementary deterministic lower bound on the supremum of co-spectral radii over all conjugates. Then we show an example that consideration of all conjugates is necessary. This example will also show the necessity of the independence assumption in Corollary \ref{cor:intersectIRS} on the co-amenability of the intersection of a pair of independent co-amenable IRS'es.
\begin{proof}[Proof of Theorem \ref{thm:DetInt}] As in the introduction, we let $M:=\frac{1}{|S|}\sum_{s\in S}s\in \mathbb C[\Gamma]$. We have the following identity between the unitary representations of $\Gamma$:
$$L^2(H_1\bs\Gamma)\otimes L^2(H_2\bs\Gamma)\simeq \bigoplus_{g\in H_1\bs \Gamma/H_2} L^2((H_1\cap H_2^g)\bs \Gamma).$$
Write $\pi_1,\pi_2$ for the unitary representations corresponding to $L^2(H_1\bs\Gamma)$ and $L^2(H_2\bs\Gamma)$. The above identity implies that 
    \[ \sup_{g\in H_1\bs \Gamma/H_2} \rho((H_1\cap H_2^g)\bs \Gamma)=\|(\pi_1\otimes \pi_2)(M)\|.\] 
To prove the theorem, it is enough to verify that $$\|(\pi_1\otimes \pi_2)(M)\|\geq \|\pi_2(M)\|=\rho(H_2\bs\Gamma).$$ Let $\varepsilon>0$. Choose unit vectors $u_1\in L^2(H_1\bs\Gamma)$ and $u_2\in L^2(H_2\bs\Gamma)$ such that $\langle \pi_1(s)u_1, u_1\rangle \geq 1-\varepsilon$ for all $s\in S$ and $\langle \pi_2(M)u_2,u_2\rangle\geq \|\pi_2(M)\|-\varepsilon$. Then 
\begin{align*}\langle (\pi_1\otimes \pi_2)(M) u_1\otimes u_2,u_1\otimes u_2\rangle =&\frac{1}{|S|}\sum_{s\in s} \langle \pi_1(s)u_1,u_1\rangle \langle \pi_2(s)u_2,u_2\rangle\\ \geq& \frac{1}{|S|}\sum_{s\in S} (1-\varepsilon)\langle \pi_2(s)u_2,u_2\rangle\\
\geq & (1-\varepsilon)(\|\pi_2(M)\|-\varepsilon).\end{align*} 
Letting $\varepsilon\to 0$ we conclude that $\|(\pi_1\otimes \pi_2)(M)\|\geq \|\pi_2(M)\|.$
\end{proof}
The supremum in the inequality seems to be necessary. Below we construct an example of a non-amenable finitely generated group $\Gamma$ with two co-amenable subgroups $H_1,H_2$ such that the intersection $H_1\cap H_2$ is trivial. In particular $$\rho(\Gamma)=\rho((H_1\cap H_2)\bs \Gamma)<\rho(H_2\bs\Gamma)=1.$$

\begin{ex}\label{ex:Wreath} Let $\Gamma := F_2^{\oplus \bbZ} \rtimes \bbZ$, where $F_2$ stands for the free group on two generators. The group is obviously non-amenable. Let $a,b$ be the standard generators of $F_2$ and let $s$ be the generator of the copy of $\bbZ$ in $\Gamma$. The triple $\{s,a,b\}$ generates $\Gamma$. Put $S:=\{s,a,b,s^{-1},a^{-1},b^{-1}\}.$ For any subset $E\subset \bbZ$ let $H_E:=F_2^{\oplus E}\subset \Gamma$. 

Now let $A,B$ be disjoint subsets of $\bbZ$ containing arbitrary long segments. Since $A\cap B=\varnothing$, the intersection $H_A\cap H_B=1$ is not co-amenable. On the other hand, we claim that for any subset $C$ containing arbitrarily long segments, $H_C$ is co-amenable, so that in particular $H_A$ and $H_B$ are co-amenable: Indeed, suppose  $C\subseteq\bbZ$ contains arbitrarily long segments. Then for any $g\in \Gamma$, the Schreier graphs for $H_C$ and $H_C^g$ are isomorphic, so $\rho(H_C\bs\Gamma)=\rho(H_C^g\bs\Gamma)$. For every $n\in\mathbb N$ 
\[\rho(H_C\bs \Gamma)=\rho(H_C^{s^n}\bs \Gamma)=\rho(H_{C-n}\bs\Gamma).\]
Let $(n_k)_{k\in\mathbb N}$ be a sequence such that $\{-k,-k+1,\ldots,k-1,k\}\subset C-{n_k}$. Then $H_{C-n_k}$ converges to $H_{\mathbb Z}$ in ${\rm Sub}(\Gamma)$ as $k\to\infty$. Since the spectral radius is lower semi-continuous on the space of subgroups, we get 
\[\rho(H_C\bs \Gamma)=\liminf_{k\to\infty}\rho(H_{C-n_k}\bs \Gamma)\geq \rho(H_{\mathbb Z}\bs \Gamma)=\rho(\mathbb Z)=1.\]
\end{ex}

\begin{ex} The above example also shows that for the intersection of two co-amenable IRS'es to be co-amenable (Corollary \ref{cor:intersectIRS}), the independence assumption is necessary. Indeed, let $\Gamma$ be as in the previous example and let $A$ be an invariant percolation on $\bbZ$ such that both $A$ and its complement contain arbitrarily long segments (e.g. Bernoulli percolation). Then $H_A$ and $H_{A^c}$ are co-amenable but their intersection is trivial.
\label{ex:independence}\end{ex} 

\section{Embedded spectral radius} \label{sec:whf}


Let us introduce some terminology. Let $(X,\nu)$ be a measure preserving $\Gamma$-action, and write $\calR$ for the corresponding orbit equivalence relation. We shall assume that $\nu$ is $\sigma$-finite but not necessarily finite. We recall that for every $x\in X$, $\mathcal G_x$ is the labeled graph with vertex set $[x]_\calR$ and edge set $(y,s y), y\in \Gamma x$, labeled by $s\in S$. 

\begin{dfn}\label{def:FinComp} A set $P\subset X$ is called a \textbf{finite connected component} if for almost all $x\in P$ the connected component of $x$ in the graph $P\cap \mathcal{G}_x$ is finite. In other words, the graphing restricted to $P$ generates a finite equivalence relation.
\end{dfn}
For any subset $P\subset X$ write $\partial P:=SP\setminus P $ for the (outer) boundary and ${\rm int}(P):=P\setminus \partial(X\setminus P)$ for the interior of $P$.

\begin{dfn}\label{def:WeakSpectralRad}
Let $(X,\nu)$ be a measure-preserving $\Gamma$-action. We say that $(X,\nu)$ has \textbf{embedded spectral radius} $\lambda$ if for every finite measure finite connected component $P\subset X$ and every $f\in L^2(X,\nu)$ supported on the interior ${\rm int}(P)$, we have
    $$\langle (I-M)f, f\rangle \geq (1-\lambda)\|f\|^2,$$ 
and $\lambda$ is minimal with this property.
\end{dfn}
\begin{rmk}
Using the monotone convergence theorem, we may assume that $f$ in the above definition is bounded. Further, taking the absolute value of $f$ leaves the right-hand side unchanged, and decreases the left-hand side. Therefore it suffices to consider nonnegative functions $f\geq 0$.
\label{rmk:nonnegative}
\end{rmk}
Our goal in this section is to prove that the embedded spectral radius of a measure preserving system $\Gamma\curvearrowright (X,\nu)$ is detected by the co-spectral radius along orbits:
\begin{prop}\label{prop:WHandCoAm} Let $(X,\nu)$ be a $\sigma$-finite measure-preserving $\Gamma$-system. Then the stabilizer of almost every point has co-spectral radius at least $\lambda$ if and only if almost every ergodic component of $\nu$ has embedded spectral radius at least $\lambda$.
\end{prop}
\begin{rmk}
This result can be used to give examples whose embedded spectral radius is strictly less than the spectral radius of $M$ on $L^2_0(X,\nu)$. This happens for example when $\Gamma$ is a non-abelian free group and $X=X_1\times X_2$ is a product of an essentially free action $X_1$ with an action $X_2$ that has no spectral gap. In this case, the graphs $\mathcal G_x$ are just copies of the Cayley graph of $\Gamma$ so their spectral radius is bounded away from $1$. On the other hand, the spectral radius of $M$ on $L^2_0(X,\nu)$ is $1$, because it contains $L^2_0(X_2,\nu_2)$.
\end{rmk}

\begin{proof}
By passing to ergodic components we can assume without loss of generality that $(X,\nu)$ is ergodic. If $(X,\nu)$ is periodic then there is nothing to do, so henceforth we will assume that $(X,\nu)$ is an aperiodic measure preserving ergodic system.

First let us prove that if the co-spectral radius of the stabilizer $\Gamma_x$ is at least $\lambda$, then $X$ has embedded spectral radius at least $\lambda$. Let $\varepsilon>0$ be arbitrary. Then $\nu$-almost every orbit $\mathcal{G}_x$ supports a function $f_x:\mathcal{G}_x\to \bbR$ such that
    \begin{equation}
        \langle (I-M)f_x,f_x\rangle \leq (1-\lambda+\varepsilon)\|f_x\|^2.
        \label{eq:rhogeqlambda}
    \end{equation}
    

Since $\mathcal{G}_x$ is countable, using the Monotone Convergence Theorem, we can assume $f_x$ is supported on a finite ball. Let $R_x>0$ be minimal such that the interior of the ball $B_{\mathcal G_x}(x,R_x)$ of radius $R_x$ around $x$ supports a function $\psi_x$ satisfying \eqref{eq:rhogeqlambda}. Since balls of fixed radius depend measurably on $x$, the map $x\mapsto R_x$ is measurable, so we can choose $R_0>0$ such that 
    $$\nu\, (\{x\in X \mid R_x\leq R_0\}) >0$$
and put $X_1:=\{x\in X \mid R_x\leq R_0\}.$ Since there are only finitely many rooted graphs of radius $R_0$ labeled by $S$, there exists a positive measure set $X_2\subset X_1$ such that for all $x\in X_2$, the rooted graphs $(B_{\mathcal G_x}(x,R_0),x)$ are all isomorphic to some $(\mathcal G,o)$ as rooted $S$-labeled graphs. By restricting to a smaller subset, we can assume that $\nu(X_2)$ is finite. Fix $\psi:\mathcal{G}\to\bbR$ satisfying
    $$\langle (I-M)\psi, \psi\rangle \leq (1-\lambda+\varepsilon)\|\psi\|^2,$$
and for $x\in X_2$, let $B_x\subset X$ be the image of $\mathcal{G}$ via the unique labeled isomorphism $(\mathcal G,o)\simeq (B_{\mathcal{G}_x}(x,R_0),x)$. 

Let $S'$ be the set of all products of at most $2R_0$ elements of $S$.  At this point we need to use a Rokhlin-type lemma, which will be stated and proved below (see Lemma \ref{lem:TurboRokhlin}). Upon applying this to the graphing $(S'X_2,\nu, S')$ we find a partition $S'X_2=B\sqcup \bigsqcup_{j=1}^N A_j$ with $\nu(B)<\nu(X_2)/2$, such that $A_j\cap sA_j=\{x\in A_j\mid sx=x\}$ for every $s\in S'$. This translates to the condition that that $B_{x}$ and $B_{x'}$ are disjoint for every distinct pair of points $x,y \in A_j$. Since the sets $A_j$ cover a subset of $X_2$ of measure at least  $\nu(X_2)/2$, there exists $j$ such that $X_3:=X_2\cap A_j$ has positive measure.  
The set $P:=\bigcup_{x\in X_3} B_x$ is then a disjoint union of its finite connected components $B_x$, so it is a finite connected component in the sense of Definition \ref{def:FinComp}. The function $\psi:\mathcal{G}\to\bbR$ naturally induces a function $f:P\to\bbR$ defined by $f|_{B_x} := \psi$ for all $x\in X_3$, where we have identified $B_x$ with $\mathcal{G}$ using the isomorphism of rooted $S$-labeled graphs $(B_x, x)\simeq (\mathcal{G},o)$.

Then we easily verify $\langle (I-M)f, f\rangle \leq (1-\lambda+\varepsilon)\|f\|^2$, namely
    $$\|f\|^2 = \int_{X_3} \|\psi\|^2 d\nu = \nu(X_3) \|\psi\|^2,$$
and similarly
    $$  \langle Mf, f\rangle = \int_{X_3} \langle M\psi,\psi\rangle \, d\nu = \nu(X_3) \langle M\psi , \psi\rangle.$$
 Taking $\varepsilon\to 0$, we see that $X$ has embedded spectral radius at least $\lambda$.

We prove the other direction. The proof will use the mass transport principle for unimodular random graphs. In our case the unimodular random graph is given by $(\mathcal G_x,x)$ where $x\in X$ is $\nu|_P$-random. We argue by contradiction, so assume that $(X,\nu)$ has embedded spectral radius at least $\lambda$ but at the same time stabilizers have co-spectral radius $\rho<\lambda$  with positive probability. By ergodicity, there exists an $h>0$ such that $\rho(\mathcal{G}_x)\leq\lambda-h$ a.s.

Since $X$ has spectral radius at least $\lambda$, there exists $f\in L^2(X,\nu)$ nonzero and supported on the interior of a finite connected component $P\subseteq X$ with $\nu(P)<\infty$ such that
\begin{equation}\label{eq:WeakHypCnd} \langle (I-M)f,f\rangle \leq \left(1-\lambda+\frac{h}{2}\right)\|f\|^2.\end{equation}

As in Section \ref{sec:relations}, let $\mathcal{R}_P$ be the equivalence relation generated by the graphing on $P$. We write $P_x^o:=[x]_{\mathcal{R}_P}$ for the connected component of $x\in P$. Define $K:\mathcal{R}_P\to\bbR$ by
    	\begin{equation}
		K(x,y):=\frac{f(x)^2}{\|f\|_{P_x^o}^2} (2|S|)^{-1} \sum_{s\in S} (f(ys)-f(y))^2
	\end{equation}
By the mass transport principle, we equate
    \begin{equation}
        \int_P \sum_{x\in P_y^o} K(x,y) \, d\nu(y) = \int_P \sum_{y\in P_x^o} K(x,y) \, d\nu(x).
        \label{eq:mtp}
    \end{equation}
We start by computing the integrand on the right-hand side. Rewriting
    $$(f(ys)-f(y))^2=f(ys)(f(ys)-f(y))+f(y)(f(y)-f(ys)),$$ 
we find (using $S$ is symmetric) for every $x\in X$ that
    \begin{align*}
        \sum_{y\in P_x^o} K(x,y)    &=  \frac{f(x)^2}{\|f\|_{P_x^o}^2} |S|^{-1} \sum_{y\in P_x^o} \sum_{s\in S} f(y)\left(f(y)-f(ys)\right)          \\
                                        &=   \frac{f(x)^2}{\|f\|_{P_x^o}^2} \left\langle (I-M)f, f\right\rangle_{P_x^o}.
    \end{align*}
Using $\rho(\mathcal{G}_x)\leq\lambda-h$ and $f\geq 0$, we can estimate
    \begin{equation}
        \sum_{y\in P_x^o} K(x,y)\geq (1-\lambda+h) f(x)^2.
        \label{eq:sumy-comp}
    \end{equation}
Therefore we have the following estimate for the right-hand side in the Mass Transport Equation \eqref{eq:mtp}
    \begin{equation}
        \int_P \left(\sum_{y\in P_x^o} K(x,y)\right) \, d\nu(x)\geq (1-\lambda+h) \|f\|^2.
        \label{eq:MTPsumy}
    \end{equation}
Next, we compute the integrand on the left-hand side of the Mass Transport Equation \eqref{eq:mtp}, namely for $y\in X$, we have
    \begin{align*}
        \sum_{x\in P_y^o} K(x,y)      &= \sum_{x\in P_y^o} \frac{f(x)^2}{\|f\|^2_{P_y^o}} (2|S|)^{-1} \sum_{s\in S} (f(ys)-f(y))^2 \\
                                       &= f(y)(I-M)(f)(y), 
    \end{align*}
where we used $S$ is symmetric and the action is measure-preserving. Hence, integrating the above equation over $y$ and using the mass transport principle to estimate this by the right-hand side of Equation \eqref{eq:MTPsumy}, we find
    $$\langle (I-M)f,f\rangle \geq (1-\lambda+h)\|f\|^2.$$
This contradicts the choice of $f$ in Equation \eqref{eq:WeakHypCnd}.
\end{proof}

We end this section with the following technical Rokhlin-type lemma that was used in the above proof:
\begin{lem}\label{lem:TurboRokhlin}
$(X,\nu,(\varphi_i)_{i\in I})$ be a finite measure preserving symmetric graphing on a finite measure space. Then, for every $\delta>0$ there exists a measurable partition $X=B\sqcup \bigsqcup_{j=0}^N A_j$, such that $\nu(B)\leq \delta$ and $$A_j \cap \varphi_i (A_j\cap U_i)= \{x\in A_j\cap U_i \mid \varphi_i(x)=x\}.$$
\end{lem}
\begin{proof}
We start by proving the lemma for a single measure preserving invertible map $\varphi\colon U\to \varphi(U).$ Since we do not assume that $\varphi$ is defined on all of $X$ we need to treat separately the subset of elements where $\varphi$ can be applied only finitely many times. For any $n\in \mathbb N$ define
$$E_n:=\{x\in X\mid \varphi^{n-1}(x)\in U \text{ but } \varphi^n(x)\not\in U\}.$$
Put $A_{\rm 0}=\cup_{n=0}^\infty E_{2n}$ and $A_{\rm 1}=\cup_{n=0}^\infty E_{2n+1}.$ We obviously have $A_0\cap A_1=\emptyset$,  $\varphi^{\pm 1}(A_0)\subset A_1$ and $\varphi^{\pm 1}(A_1)\subset A_0$. This reduces the problem to the subset $Y:=X\setminus \bigcup_{n=0}^\infty E_n$. By definition, $\varphi(Y)=Y$. We further decompose $Y$ into the periodic and aperiodic parts $Y^{\rm p}, Y^{\rm ap}$. The periodic part can be partitioned into a fixed point set $A_2=\{x\in Y^{\rm p}\mid \varphi(x)=x\}$, finitely many sets $A_3,\ldots, A_M$ permuted by $\varphi$ and a remainder $B_1$ of measure $\nu(B_1)<\delta/2$ coming from large odd periods. By the usual Rokhlin lemma, the aperiodic part $Y^{\rm ap}$ can be decomposed as $Y^{\rm ap}=B_2\sqcup A_{M+1}\sqcup\ldots \sqcup A_N$ where $\nu(B_2)<\delta/2$, $\varphi(A_{M+k})=A_{M+k+1}$ for all $M+k<N$ and $\varphi(A_N)\subset B_2.$ Put $B=B_1\cup B_2$. This ends the construction for a single map. 

Suppose now that the graphing consists of $d$ maps $\varphi_1,\ldots,\varphi_d$ and their inverses. For each $i=1,\ldots d$ there exists a partition $X=B^i\sqcup \bigsqcup_{j=1}^{N_i}A_j^i$ such that $\nu(B^i)<\frac{\delta}{d}$ and $\varphi_i (A_j^i\cap U_i)\cap A_j^i= \{x\in (A_j^i\cap U_i) \mid \varphi_i(x)=x\}.$
Let $B:=\cup_{i=1}^d B^i$ and define the partition $\{A_j\}$ as the product partition $\bigwedge_{i=1}^d  \{A_j^i\}.$ This partition satisfies all the desired conditions.
\end{proof}

\section{Proof of the Main Theorem}

\subsection{Preliminary reductions and general strategy} Let $\Gamma$ be a countable group with a co-amenable subgroup $H_1$ and an IRS $H_2$ with co-spectral radius $\lambda_2$. We need to show that $H_1\cap H_2$ has co-spectral radius at least $\lambda_2$ as well. First of all, without loss of generality we can assume $H_2$ is ergodic. The group $H_1$ is realized as the stabilizer of a point in $X_1:=H_1\bs \Gamma$ and $H_2$ is realized as the stabilizer of a random point in a p.m.p. action of $\Gamma$ on $(X_2,\nu_2)$. We use Proposition \ref{prop:WHandCoAm} to find a finite connected component $P_2$ of $X_2$ and a function $f_2$ on $P_2$ that witnesses the spectral radius $\lambda_2$. Next, using a large F\"olner set in $X_1$, we produce a new finite connected component in the product system $X_1\times X_2$ and a new function which certifies that the co-spectral radius of stabilizers in $X_1\times X_2$ is arbitrarily close to $\lambda_2$ on almost every ergodic component of the product measure.  

\subsection{Reformulation of the problem in measure theoretic terms.}\label{sec:MSetup}

Write $(X_1,\nu_1)$ for the set $H_1\backslash \Gamma$ endowed with the counting measure. It is an infinite ergodic measure-preserving action of $\Gamma$. Let $\Gamma\curvearrowright (X_2,\nu_2)$ be a p.m.p. Borel action on a standard Borel probability space such that $H_2=\Gamma_x$ for $\nu_2$-random $x$.  We will consider the action of $\Gamma$ on the product system $(X_1\times X_2, \nu_1\times \nu_2).$ To shorten notation we write $\nu=\nu_1\times \nu_2$. The intersection $H_1\cap H_2$ is nothing else than the stabilizer of a random point $x\in \{[H_1]\}\times X_2$. Note that for such $x$, we have $\rho(\Gamma_x\bs\Gamma)\leq \lambda_2:=\rho(H_2\bs\Gamma)$ almost surely. Set 
$$C_0:=\{x\in X_1\times X_2 \mid  \rho(\Gamma_x \bs \Gamma)= \lambda_2\}.$$
	
Since conjugate subgroups have the same co-spectral radius, the set $C_0$ is invariant under the action of $\Gamma$.
Let $(X_1\times X_2, \nu)\to (Z,\tau)$ be the ergodic decomposition given by Corollary \ref{cor:erg-decomp}, and set \[Z_0:=\{z\in Z\mid \nu_z(C_0)>0\}.\] By ergodicity and invariance, the set $C_0$ has full $\nu_z$-measure for every $z\in Z_0$.  Theorem \ref{thm:main-spectral} is equivalent to the identity $C_0=X_1\times X_2$ modulo a null set, so it will follow once we show that  $\tau(Z_0)=1$. By Proposition \ref{prop:WHandCoAm}, $z\in Z_0$ if and only if the following condition holds: For every $\eta>0$ there exists a function $h$ supported on the interior of a finite connected component of $(X_1\times X_2,\nu_z, S)$ (according to Definition \ref{def:FinComp}), such that 
\begin{equation}\label{eq:Condition}
    \langle (I-M)h, h\rangle_{\nu_z}\leq (1-\lambda_2+\eta)\|h\|_{\nu_z}^2.
\end{equation}
We will refer to nonnegative, nonzero functions supported on interiors of finite connected components of $(X_1\times X_2,\nu, S)$ as \textbf{test functions}. It is easy to check that a test function for $\nu$ is also a test function for almost all ergodic components $\nu_z$. It will be convenient to name the set of ergodic components $z$ for which there exist a test function satisfying \eqref{eq:Condition} with specific $\eta$. Let 
$$Z_\eta:=\{z\in Z\mid \text{ there exists } h \text{ such that } \langle (I-M)h, h\rangle_{\nu_z}\leq (1-\lambda_2+\eta)\|h\|_{\nu_z}^2\}.$$
Obviously we have $Z_0=\bigcap_{\eta>0} Z_\eta$ and $Z_{\eta}\subset Z_{\eta'}$ for $\eta<\eta'$. In the following sections we show that $\tau(Z_{\eta})\to 1$ as $\eta\to 0$. This will imply that $\nu(Z_\eta)=1$ for every $\eta>0$, and consequently that $\nu(Z_0)=1$, which is tantamount to Theorem \ref{thm:main-spectral}.

\subsection{Construction of test functions.} 
\begin{lem}\label{lem:TestF}
Let $\delta>0$. There exists a test function $f$ and a set $Z'\subset Z$ such that 
\begin{enumerate}
    \item $\|f\|_{\nu_z}^2\geq (1-\delta)\|f\|_{\nu}^2$, for every $z\in Z'$.
    \item $\tau(Z')\geq 1-\delta.$
    \item $\langle (I-M)f,f\rangle_{\nu} \leq (1-\lambda_2+\delta)\|f\|_{\nu}^2.$
\end{enumerate}
\end{lem}
\begin{proof}
Let  $\varepsilon_2>0$. Since $\Gamma\curvearrowright (X_2,\nu_2)$ has embedded spectral radius $\lambda_2$, there is a finite measure, finite connected component $P_2\subset X_2$ and nonzero $f_2\in L^2(X_2,\nu_2)$ as is in  Definition \ref{def:WeakSpectralRad}, i.e. $f_2$ is supported on the interior ${\rm int}(P_2)$ and 
    \begin{equation}\label{eq:f2}\langle (I-M)f_2,f_2\rangle\leq (1-\lambda_2+\varepsilon_2)\|f_2\|^2.\end{equation}
By Remark \ref{rmk:nonnegative}, we may assume $f_2\geq 0$.

We will show that for a good enough F{\o}lner set $F\subseteq X_1$ and small enough $\varepsilon_2$, the function $$f:=\mathds{1}_{F} \times f_2$$ satisfies the conditions of the lemma. While (3) is relatively straightforward, conditions (1) and (2) require some work and strongly use the fact that $X_2$ is a p.m.p. action. 

Consider the following probability measures on $\Gamma$: 
    $$\mu:=\frac{1}{|S|}\sum_{s\in S}\delta_s\textrm{ and } \mu^m:= \frac{1}{m}\sum_{i=0}^{m-1}\mu^{\ast i}\textrm{ for } m\in\bbN.$$ 
By Kakutani's ergodic theorem \cite{Kakutani1951}, there exists $m_0\geq 1$ such that 
\begin{equation}\label{eq:Kakutani} \left|\int_{\Gamma}f_2^2(x\gamma^{-1})d\mu^{m_0}(\gamma)-\int_{X_2} f_2^2 d\nu_2\right|\leq \varepsilon_2\|f_2\|^2
\end{equation} for all $x\in X_2'$ where $\nu_2(X_2')\geq 1-\varepsilon_2$.

Fix $0<\varepsilon_1\ll \varepsilon_2$ very small. The precise choice only depends on $\varepsilon_2$ and will be specified at the end of the proof. Let $F\subset X_1$ be an $\varepsilon_1$-F{\o}lner set and write $Y=(F\cup \partial F)\times X_2$. 

Write $F'$ for the set of points of $F$ which are at distance at least $m_0$ from the boundary $\partial F$ and set $Y':=F'\times X_2'$, where $X_2'$ is as in (\ref{eq:Kakutani}). We claim that for $\varepsilon_1$ small enough, we will have $(\nu_1\times \nu_2)(Y')\geq |F|(1-2\varepsilon_2)$. Indeed, using that $F$ is $\varepsilon_1$-F{\o}lner, we have
	$$|F'|\geq |F\cup\partial F|-|\partial F|\sum_{i=1}^{m_0-1}|S|^{i}\geq |F\cup \partial F|(1-|S|^{m_0}\varepsilon_1).$$ 
Clearly for sufficiently small $\varepsilon_1$, we have 
	\begin{equation}\label{eq:YprimLarge}
		\nu(Y')=|F'|\nu_2(X'_2)\geq (1-\varepsilon_1|S|^{m_0})(1-\varepsilon_2)|F|\geq (1-2\varepsilon_2)|F|.
	\end{equation}
Write $P:=(F\cup\partial F)\times P_2\subset Y$. By construction, the support of $f$ is contained in $P\subset Y$. Note that since $P_2$ is a finite connected component of $X_2$ and $F\cup\partial F$ is finite, the set $P$ will be a finite connected component of $(X_1\times X_2,\nu,S)$ in the sense of the Definition \ref{def:FinComp}.  
Let $\nu=\int_Z\nu_zd\tau(z)$ be the ergodic decomposition of $\nu$ as in Section \ref{sec:decomp}. The set $P$ is also a finite connected component of $(X_1\times X_2,\nu_z,S)$ for almost every $z\in Z$.
For every $z\in Z$, the measure $\nu_z$ is invariant under the action of $\Gamma$, so that 
    \begin{align*}
    \int f^2(x_1,x_2)d\nu_z(x_1,x_2)&= \int_\Gamma \int f^2(x_1\gamma^{-1},x_2\gamma^{-1}) d\nu_z(x_1,x_2) d\mu^{m_0}(\gamma)\\   
    &\geq \int_{Y'}\int_\Gamma f^2(x_1\gamma^{-1},x_2\gamma^{-1}) d\mu^{m_0}(\gamma)d\nu_z(x_1,x_2)
    \end{align*}
Since $Y'=F'\times X_2'$ and $F'\gamma^{-1}\subset F$ for any $\gamma\in \supp \mu^{m_0}$ we can use the identity $f=\mathds 1_F \times f_2$ to rewrite the last integral as 
\[ \int_{F'\times X_2'} \left(\int_\Gamma f_2^2(x_2\gamma^{-1})d\mu^{m_0}(\gamma)\right)d\nu_z(x_1,x_2).\]
We use (\ref{eq:Kakutani}) to estimate the innermost integral and obtain a lower bound on $\|f\|_{\nu_z}^2$:
	\begin{align*}
		\int_P f^2 d\nu_z\geq & (1-\varepsilon_2)\nu_z(F'\times X_2')\int_{X_2} f_2^2 d\nu_2 = (1-\varepsilon_2) \nu_z(Y')\|f_2\|_{\nu_2}^2.
	\end{align*}
By \eqref{eq:YprimLarge}, we have $\nu(Y')\geq (1-2\varepsilon_2)|F|$, so we can apply Markov's inequality to get a set $Z'\subset Z$ with $\tau(Z')\geq 1-\sqrt{2\varepsilon_2}$ such that $\nu_z(Y')\geq (1-\sqrt{2\varepsilon_2})|F|$ for $z\in Z'$. Finally we get that for $z\in Z'$:
	\begin{equation}\label{eq:LBonP}
		\|f\|_{\nu_z}^2\geq (1-\varepsilon_2)(1-\sqrt{2\varepsilon_2})|F|\|f_2\|_{\nu_2}^2=(1-\varepsilon_2)(1-\sqrt{2\varepsilon_2})\|f\|_{\nu}^2.
	\end{equation}
This establishes Properties (1) and (2) of the lemma. It remains to address (3). 
Using $f_2\geq 0$, we estimate $\langle f,Mf\rangle_\nu$ in terms of $\langle f_2, M f_2\rangle_{\nu_2}$ as follows:
    \begin{align*}
        \langle f,Mf\rangle_{\nu} &= |S|^{-1} \int_{X_1} \int_{X_2} \bbone_F(x_1) f_2(x_2) \sum_{s\in S} \bbone_F(x_1 s) f_2(x_2 s) d\nu_2(x_2) d\nu_1(x_1)   \\
                    &\geq |S|^{-1} \int_{\text{int}(F)} \int_{X_2} f_2(x_2) \sum_{s\in S} f_2(x_2 s)       d\nu_2(x_2) d\nu_1(x_1)        \\
                    &= |\text{int}(F)| \langle f_2, M f_2\rangle_{\nu_2}.
    \end{align*}
Hence
    \begin{align*}
        \langle (I-M)f,f\rangle_{\nu}   &\leq |F|\langle f_2,f_2\rangle_{\nu_2} -|\text{int}(F)|\langle f_2, Mf_2\rangle_{\nu_2}        \\
                        &=|F|\langle f_2-M f_2,f_2\rangle_{\nu_2} +|F|\left(1-\frac{|\text{int}(F)|}{|F|}\right)\langle Mf_2, f_2\rangle_{\nu_2}.
    \end{align*}
Since $F$ is $\varepsilon_1$-F{\o}lner, we have $|\text{int}(F)|/|F|\geq 1-|S|\varepsilon_1$, so finally we obtain
    \begin{align*}  \langle (I-M)f,f\rangle_{\nu}\leq&  |F|\left(\langle f_2-Mf_2,f_2\rangle_{\nu_2}+|S|\varepsilon_1 \langle Mf_2, f_2\rangle_{\nu_2}\right)\\
    \leq&|F|\left(\langle f_2-Mf_2,f_2\rangle_{\nu_2}+|S|\varepsilon_1\|f_2\|_{\nu_2}^2\right).\end{align*}
By the defining property of $f_2$ (see Equation \eqref{eq:f2}), we then find
    $$\langle (I-M)f,f\rangle_{\nu}\leq |F|(1-\lambda_2+\varepsilon_2+|S|\varepsilon_1)\|f_2\|_{\nu_2}^2=(1-\lambda_2+\varepsilon_2+|S|\varepsilon_1)\|f\|_{\nu}^2.$$
To finish the proof, choose $\varepsilon_2>0$ such that $(1-\varepsilon_2)(1-\sqrt{2\varepsilon_2})\geq (1-\delta)$ and then choose $\varepsilon_1>0$ such that $\varepsilon_2+|S|\varepsilon_1\leq\delta$ and such that \eqref{eq:YprimLarge} holds.
\end{proof}
\subsection{End of the proof.} 

Let $f, \delta, Z'$ be as in the Lemma \ref{lem:TestF}. Using Property (3) of the Lemma, we have \begin{align*}
    \int_{Z'}\langle (I-M)f,f\rangle_{\nu_z}d\tau(z)\leq& \int_{Z}\langle (I-M)f,f\rangle_{\nu_z}d\tau(z)\\ =&\langle (I-M)f,f\rangle_\nu \\
                                    \leq& (1-\lambda_2+\delta)\|f\|_\nu^2.
\end{align*}
By Property (1), we can then estimate
\[  \int_{Z'}\frac{\langle (I-M)f,f\rangle_{\nu_z}}{\|f\|_{\nu_z}^2}d\tau(z)\leq\frac{1-\lambda_2+\delta}{1-\delta}.\]
On the other hand, Proposition \ref{prop:WHandCoAm} and the fact that co-spectral radii of stabilizers are all at most $\lambda_2$, yield the inequality $$\frac{\langle (I-M)f,f\rangle_{\nu_z}}{\|f\|_{\nu_z}^2}\geq 1-\lambda_2 \text{ for almost all }z\in Z.$$ Therefore, by  Markov's inequality there is a positive $\eta=O(\sqrt\delta)$ and a subset $Z''\subset Z$ such that $\tau(Z'')\geq 1-\eta$ and 
$$ \langle (I-M)f,f\rangle_{\nu_z}\leq (1-\lambda_2+\eta){\|f\|_{\nu_z}^2} \text{ for } z\in Z''.$$ 
$\eta$ could be made explicit in terms of $\delta$ and $\lambda_2$, but we will only need that $\eta\to 0$ as $\delta\to 0$. We have $Z''\subset Z_\eta$, so it follows that $\tau(Z_\eta)\to 1$ as $\eta\to 0$. This ends the proof per the discussion at the end of Section \ref{sec:MSetup}. \qed
\bibliographystyle{plain}
\bibliography{ref}
\end{document}